\documentclass[11pt]{article}
\usepackage{graphicx, color}
\usepackage{bm}
\usepackage{amsmath,latexsym,amsbsy,amssymb}
\def\ds{\displaystyle}
\def\forall{\hbox{for all}~}
\def\L{{\bf L}}
\def\H{{\cal H}}

\def\ve{\varepsilon}

\def\dint{\int\!\!\int}

\def\E{{\cal E}}
\def\I{{\cal I}}
\def\J{{\cal J}}
\def\S{{\cal S}}

\def\wto{\rightharpoonup}
\def\R{{\mathbb R}}

\def\P{{\cal P}}
\def\IPC{{\cal IP}_C}

\def\vs{\vskip 2em}

\def\v{\vskip 1em}

\def\begi{\begin{itemize}}
\def\endi{\end{itemize}}

\def\TP{{\cal IP}}
\def\ov{\overline}
\def\Tilde{\widetilde}
\def\Hat{\widehat}
\def\bega{\begin{array}}
\def\enda{\end{array}}
\def\meas{\hbox{meas}}

\def\bel{\begin{equation}\label}
\def\eeq{\end{equation}}
\def\sqr#1#2{\vbox{\hrule height .#2pt
\hbox{\vrule width .#2pt height #1pt \kern #1pt
\vrule width .#2pt}\hrule height .#2pt }}
\def\square{\sqr74}
\def\endproof{\hphantom{MM}\hfill\llap{$\square$}\goodbreak}

\newtheorem{theorem}{Theorem}[section]

\newtheorem{proposition}[theorem]{Proposition}

\newtheorem{definition}[theorem]{Definition}
\newtheorem{remark}[theorem]{Remark}
\newtheorem{example}[theorem]{Example}

\begin{document}
\title{\bf A Mollification Approach to Ramified Transport and Tree Shape Optimization}
\vs

\author{Alberto Bressan,  Giacomo Vecchiato and Ludmil Zikatanov\\
\, \\
Department of Mathematics, Penn State University, \\
University Park,
Pa.~16802, USA.\\
\,\\
E-mails: axb62@psu.edu,~giacvecchiato@gmail.com,~ludmil@psu.edu
}
\maketitle
\begin{abstract} The paper analyzes a mollification algorithm for  the numerical computation of optimal irrigation patterns. This provides a regularization of the 
standard irrigation cost functional in a Lagrangian framework.  Lower semicontinuity and Gamma-convergence results are proved.
The technique is then applied
to some numerical optimization problems related to the optimal shape of tree roots and branches.
\end{abstract}

\section{Introduction}
\label{s:inro}
\setcounter{equation}{0}
Aim of this paper is to analyze a numerical algorithm for optimal ramified transport, with applications
to some shape optimization problems for tree roots and branches.

Problems of optimal transport with concave power law were first proposed by 
Gilbert \cite{G}, in a discrete framework modeling communication networks. 
Irrigation problems for general Radon measures were formulated in the seminal papers  
\cite{MMS, X03}, providing the framework for all subsequent mathematical literature.  See
\cite{BCM} for a comprehensive introduction and further references.

The  paper \cite{BSun} introduced a class of geometric optimization problems, 
modeling the optimal shape of tree roots and branches.   The
functionals to be maximized include a payoff, modeling the amount of sunlight captured by 
the distribution of leaves, and the amount of water and nutrients absorbed by the roots, together
with a ramified irrigation cost.
The existence of optimal solutions, together with various 
qualitative properties, were studied in \cite{BGS, BPS, BSun, BSun2}.
However, 
due to the lack of continuity and convexity of the ramified transportation cost, the numerical computation of optimal 
irrigation patterns has remained a difficult problem.  

In this direction,  it is natural  to approximate the irrigation cost with more regular cost functionals. 
This allows to compute 
directions of steepest descent, and possibly reduces the number of local minima where a gradient descent
algorithm may get stuck.

In \cite{OS, S2} a new approach was proposed, where the singular 1-dimensional rectifiable vector measures describing the optimal flow are approximated by more regular vector fields, 
minimizing  a sequence of elliptic energies. 
These can be regarded as a regularization of the irrigation functional, in an Eulerian setting.

In the present paper we explore an alternative kind of regularization, in the Lagrangian setting.
Based on the Lagrangian formulation \cite{BCM, BS}, we consider two types of mollifications of the  multiplicity functional.  

More in detail: Section~\ref{sec:2} provides a brief review of optimal irrigation, introducing  the notation that will be used in the remainder of the paper.  
In Section~\ref{sec:3} we introduce a mollified multiplicity functional and study the corresponding mollified irrigation cost. The main result shows that 
this approximated cost is lower semicontinuous, hence it admits a global minimizer.  In Section~\ref{sec:4} we prove that, 
as  the mollification parameter $\ve\to 0+$, the mollified irrigation costs $\Gamma$-converge to the original irrigation cost, thus validating the numerical approach. 
In Section~\ref{sec:5} we consider an alternative mollification approach, where a maximum
is replaced by an integral average, thus having better regularity properties.
% This was originally motivated by the expectation that this additional regularity may improve the performance of gradient-descent algorithms.     
  In this direction, mixed results are obtained. A counterexample shows that in this case the mollified cost functionals are not lower semicontinuous, hence the existence of minimizers cannot be guaranteed.
On the other hand, as $\ve\to 0+$, we still have the $\Gamma$-convergence of the
the mollified irrigation costs to the original cost.
Section~\ref{sec:7} briefly reviews some shape optimization problems, related to tree roots and branches \cite{BSun}.     A simplified payoff functional, which still captures some key features of the problem, is introduced in (\ref{Hpen}).

The remaining two sections illustrate various numerical simulations, based on our mollification approach. 
Section~\ref{sec:6} is concerned with a discretized version of 
the optimal irrigation problem, for a measure uniformly distributed on a half-circumference.
It is here apparent that minimizers obtained with different mollification parameters converge to  the optimal irrigation pattern for the original problem.
Finally,  in  Section~\ref{sec:8}   we  show the results of several numerical simulations,
related to a shape optimization problem for tree roots or tree branches.

For further properties of optimal irrigation patterns we refer to \cite{DS07, DS, MMS, MS, MoS, PSX, S1, X15}.
Optimization problems for tree roots and branches have also been studied in \cite{BG, V}.
\v
\section{Review of optimal irrigation patterns}
\label{sec:2}
\setcounter{equation}{0}
Let $\mu$ be a positive Radon measure on $\R^d$ with bounded support 
and total mass $M = \mu(\R^d)$. 
Set $\Theta = [0,M]$. We think each $\theta \in \Theta$ as a ``water particle" to be transported  from the origin to 
various locations in $\R^d$.
 Given $\alpha \in [0,1]$,  following the Lagrangian approach in \cite{MMS}, 
the $\alpha$-irrigation cost of $\mu$ can be defined as follows.

\begin{definition} {\bf (irrigation plan).} \label{d:21}
A measurable map 
\bel{iplan}
\chi:\Theta\times \R_+~\mapsto~ \R^d\eeq
is called an {\bf irrigation plan}
if 
\begi
\item[(i)] At time $t=0$ all particles are at the origin:
$\chi(\theta,0)={\bf 0}$ for all $\theta\in\Theta$.
\item[(ii)] For each $\theta\in \Theta$, the map
$t\mapsto \chi(\theta,t)$ is 1-Lipschitz continuous and eventually constant.
More precisely, for each $\theta$ there exists a stopping time 
\bel{stopT} T(\theta)~\doteq~\inf~\bigl\{ t>0\,;~~\chi(\theta, t')=\chi(\theta,t)~~\forall~t'\geq t\bigr\}\eeq
such that the following holds.   Denoting by
$$\dot \chi(\theta,t)~=~{\partial\over\partial t} ~\chi(\theta,t)$$
the partial derivative w.r.t.~time, one has
\bel{stime}\left\{ \bega{rl} \bigl|\dot \chi(\theta,t)\bigr|~\leq~1\qquad &\quad\hbox{for a.e.}
~t\in [0, T(\theta)],\\[3mm]
\bigl|\dot \chi(\theta,t)\bigr|~=~0\qquad &\quad\hbox{for}
~t\geq T(\theta).\enda\right.\eeq

\endi
Given a constant $C>0$, we denote by $\IPC$ the family of all irrigation plans 
whose stopping time satisfies
\bel{stopC}
\int_{\Theta} T(\theta) \, d\theta ~\le~ C.
\eeq
\end{definition}

\begin{definition}\label{d:210}  
In the above setting, we say that $\chi$ is an  {\bf admissible irrigation plan for the measure $\mu$}
if the push-forward of the Lebesgue measure on $[0,M]$ through the map $\theta\mapsto 
\chi(\theta,T(\theta))$ coincides with $\mu$.
In other words, for every open set $A\subset\R^d$ there holds
\bel{chi1}\mu(A)\,=\,\hbox{meas}\Big( \bigl\{ \theta\in \Theta\,;~~\chi(\theta,T(\theta))\in A\bigr\}\Big).\eeq
We denote by $\TP(\mu)$ the family of all admissible irrigation plans for the measure $\mu$.
\end{definition}

%One can think of $\chi(\theta,t)$ as the 
%position of the water particle $\theta$ at time $t$.\v
Next, to define the corresponding transportation cost, one must 
take into account the fact that, if many paths go through the same pipe, their cost decreases.   With this in mind, 
given a point $x\in \R^d$ we first compute 
how many paths go through the point $x$.  
This is described by
\bel{chi}|x|_\chi~=~\meas\Big(\bigl\{\theta\in \Theta\,;~~\chi(\theta,t)= x~~~\hbox{for some}~~t\geq 0\bigr\}\Big).\eeq
We think of $|x|_\chi$ as the {\it total flux going through the
point $x$}.
\v
\begin{definition} {\bf (irrigation cost).} \label{d:22}
For a given $\alpha\in [0,1]$,
the total cost of the irrigation plan $\chi$ is
\bel{TCg}
\E^\alpha(\chi)~\doteq~\int_\Theta\left(\int_{\R_+} \bigl|\chi(\theta,t)
\bigr|_\chi^{\alpha-1} \cdot |\dot \chi(\theta,t)|\, dt\right)
d\theta.\eeq
The  {\bf $\alpha$-irrigation cost} of a measure $\mu$
is defined as 
\bel{Idef}\I^\alpha(\mu)~\doteq~\inf_{\chi\in  \TP(\mu)} \E^\alpha(\chi).\eeq
\end{definition}

\begin{remark}\label{r:par} {\rm The multiplicity function $|x|_\chi$ at (\ref{chi}) as well as the irrigation cost (\ref{Idef})
do not change under a re-parameterization of the paths $t\mapsto \chi(\theta,t)$.
In particular, every irrigation plan can be parameterized by arc length, so that for a.e.~$t\geq 0$ one has
\bel{alp}\bigl|\dot{\chi}(\theta,t)\bigr| ~=~\left\{\bega{rl} 1\qquad &\hbox{if}~~0<t<T(\theta),\\[1mm]
0\qquad &\hbox{if}~~t>T(\theta).\enda\right.\eeq
}\end{remark}

\begin{remark} {\rm In the case $\alpha=1$, the expression (\ref{TCg}) reduces to
$$
\E^\alpha(\chi)~=~\int_\Theta\left(\int_{\R_+} |\dot \chi_t(\theta,t)|\, dt\right)
d\theta~=~\int_\Theta\bigl[\hbox{total length of the path} ~\chi(\theta,\cdot)\bigr]\, d\theta\,.
$$
Of course, this length is minimal if every path $\chi(\cdot,\theta)$
is a straight line, joining the origin with $\chi(\theta, T(\theta))$.  Hence
\bel{OIP}
\I^\alpha(\mu)~\doteq~\inf_\chi \E^\alpha(\chi)~=~\int_\Theta |\chi(\theta, T(\theta))|\, d\theta~=~\int |x|\, d\mu\,.
\eeq

On the other hand, when $\alpha<1$, moving along a path which is traveled by few other particles
comes at a  high cost. Indeed, in this case the factor $\bigl|\chi(\theta,t)
\bigr|_\chi^{\alpha-1}$ becomes  large.     To reduce the total cost,  is thus convenient 
that particles initially travel along a same path as far as possible,
 and only later their
trajectories bifurcate, eventually reaching different destinations.}
\end{remark}

Let $\gamma_0, \gamma_1 : \R_+ \to \R^d$ be two 1-Lipschitz continuous functions. We define the distance $d$ between $\gamma_0$ and $\gamma_1$ as:
$$
d(\gamma_0,\gamma_1)~\doteq~\sup_{k\ge1}~\frac{1}{k}\|\gamma_0 - \gamma_1\|_{\L^{\infty}([0,k])}\,.
$$
\begin{definition} \label{d:iplan_conv}
We say that a sequence of irrigation plans $\{ \chi_n \}_{n\geq 1}$ converges to an irrigation plan $\chi$, and write $\chi_n \to \chi$, if
$$
\lim_{n \to \infty} d\bigl(\chi_n(\theta),\chi(\theta)\bigr) = 0 \qquad \text{for a.e. } \theta \in \Theta. 
$$
\end{definition}

In the above setting, we recall a useful compactness property of irrigation plans.
\begin{theorem}\label{t:compactness}
{\bf (compactness)}
Let $(\chi_n)_{n\geq 1}$ be a bounded sequence of irrigation plans, with $\chi_n\in \IPC$ for come constant $C$ and all $n\geq 1$.    Then 
 there exists an irrigation plan $\chi$ such that (by possibly taking a subsequence and relabeling) $\chi_n \to \chi$ as $n\to\infty$.
 
 Moreover, if $\chi_n\in \TP(\mu_n)$ and $\mu\in \TP(\mu_n)$, then as $n\to \infty$ the weak convergence of measures holds:
 \bel{weaconv}
 \mu_n~\wto~\mu.\eeq
\end{theorem}
For the proof, see Theorem~3.28 in \cite{BCM}, p.\,33.

\v
For the basic theory of ramified transport we refer to \cite{BS, MMS, X03, X15}, or to the monograph \cite{BCM}. 
\v

\section{A mollified Lagrangian approach}
\label{sec:3}
\setcounter{equation}{0}
For a given positive Radon measure $\mu$ on $\R^d$, determining  an optimal irrigation plan is not an easy task, because the functional to be minimized is neither continuous, nor convex.

In this section we propose a computational approach based on the Lagrangian representation,
using a suitable mollification of the multiplicity function.
Let $J:\R_+\mapsto \R_+$ be a non-increasing, Lipschitz function such that 
\bel{Jprop} J(0) = 1,\qquad\qquad \lim_{r\to+\infty} J(r)=0.
\eeq
In particular, for computational purposes one could take
\bel{Jex}
J(r)=e^{-r},\qquad\qquad J(r)= {1\over 1+r}\,,\qquad\hbox{or}\qquad  
J(r)~\doteq~\left\{ \bega{cl} 1- r\qquad &\hbox{if} ~~ 0\leq r\leq 1,\\[1mm]
0\qquad &\hbox{if} ~~r>1.\enda\right.\eeq
We then define the rescaled functions
\bel{Jep}J_\ve(r)\,\doteq\, J(r/\ve).\eeq

Next, let $\chi:\Theta\times \R_+\mapsto \R^d$ be an admissible irrigation plan
for a measure $\mu$, as in Definition~\ref{d:21}.
For any $x\in\R^d$ we then define a {\bf mollified multiplicity} function by setting
\bel{mmi}
%|x|_{\strut \chi_\ve^\infty}~
W_\ve(x,\chi)~\doteq~\int_\Theta \max_{t\geq 0}\,J_\ve\Big(\bigl| \chi(\theta,t)-x\bigr|
\Big)\, d\theta~=~\int_\Theta J_\ve\Big(\min_{t\geq 0}\,\bigl| \chi(\theta,t)-x\bigr|
\Big)\, d\theta.\eeq
%Moreover, assuming that 
%\bel{bigi}
%\int_0^{+\infty} J(r)\, dr~\geq~{1\over 2}\,,\eeq
%we also define
%\bel{mm1}
%%|x|_{\strut \chi_\ve^1}
%W_\ve^1(x,\chi)~\doteq~\int_\Theta\min\left\{ 1, \int_0^{+\infty} {2\over \ve} \, J_\ve\Big(\bigl| \chi(\theta,t)-x\bigr| \Big)\cdot |\dot \chi(\theta,t)|\
%\, dt \right\}d\theta.\eeq
%If there is no possibility of confusion, we write $W_\ve^\infty(x) = W_\ve(x,\chi)$ or $W_\ve^1(x) = W_\ve(x,\chi)$. 
 The corresponding mollified irrigation costs are then computed by replacing (\ref{TCg}) with
\bel{TCep}
\E^\alpha_\ve(\chi)~\doteq~\int_\Theta \left(\int_{\R_+} \bigl[ W_\ve(\chi(\theta,t),\chi)
\bigr]^{\alpha-1} \cdot |\dot \chi(\theta,t)|\, dt\right)
d\theta.\eeq

{}From the assumption $J(0)=1$ it immediately follows
\bel{i10}W_\ve(x,\chi)~\geq~\meas\Big(\bigl\{\theta\in \Theta\,;~~\chi(\theta,t)= x~~~\hbox{for some}~~t\geq 0\bigr\}\Big)~=~
|x|_\chi\,,\eeq
hence 
$$\E^\alpha_\ve(\chi)~\leq~\E^\alpha(\chi).$$
%Here $W_\ve(x)$ denotes either one of the mollified multiplicities in (\ref{mmi})-(\ref{mm1}).
We observe that neither this mollified multiplicity, nor the total cost of the transportation 
plan (\ref{TCep}), are affected by a re-parameterization of the paths $t\mapsto \chi(\theta,t)$.

In the remainder of this section we prove the existence of a global minimizer for the mollified
irrigation cost, using a direct approach. Toward this goal, we first prove a lower semicontinuity result,
similar to Proposition 3.24 in \cite{BCM}, in terms of mollified multiplicities. In the following, $T_n (\theta)$ and $T(\theta)$ denote the stopping times for the irrigation plans $\chi_n$ and $\chi$ respectively.

\begin{proposition}\label{p:31}
Let $(x_n)_{n\geq 1}$ be a sequence of points in $\R^d$ and let $( \chi_n)_{n\geq 1}$  a sequence of irrigation plans such that
\bel{um1}
x_n \to x, \quad \chi_n \to \chi \quad \text{as } n \to \infty,
\eeq 
for some $x \in \R^d$ and some irrigation plan $\chi$. In addition, assume that the corresponding stopping times satisfy 
\bel{um2}
\int_{\Theta} T_n(\theta) \, d\theta ~\le~ C
\eeq
for some constant $C > 0$ and for all $n \ge 1$. Then the mollified multiplicities \eqref{mmi} satisfy 
\bel{um3}
\limsup_{n \to \infty} W_\ve(x_n,\chi_n)~ \le~ W_\ve(x,\chi).
\eeq
\end{proposition}
{\bf Proof.} {\bf 1.} For any $\lambda > 0$, consider the sets
$$
\Theta_n \,\doteq\, \bigl\{\theta \in \Theta \, ; \, T_n(\theta) > \lambda \bigr\}, \quad\qquad \Theta(\lambda)
\, \doteq\, \bigl\{\theta \in \Theta \, ; \, T(\theta) > \lambda \bigr\}. 
$$
By \eqref{um2} it follows
$$
\meas\Big( \Theta_n(\lambda) \Big) \,\le\, \frac{C}{\lambda}.
$$
Therefore, for any given $\delta > 0$, choosing $\lambda > C / \delta$ one obtains
\bel{pum1}
\meas\Big( \Theta_n(\lambda) \Big)\, \le\, \delta
\eeq
for every $n \ge 1$. For a.e.~$\theta \in \Theta$ one has (see Lemma 3.20 in \cite{BCM})
\bel{liT}
\liminf_{n \to \infty}~ T_n(\theta)\, \ge \,T(\theta).
\eeq
Hence (up to a set of measure zero)
the complementary sets $\Theta_n^c (\lambda) = \Theta \setminus \Theta_n(\lambda)$ and $ \Theta^c(\lambda) = \Theta \setminus \Theta(\lambda) $ satisfy 
\bel{pum2}
\bigcap_{k\geq 1} \bigcup_{n \ge k} \Theta_n^c(\lambda) ~\subseteq~ \Theta^c(\lambda). 
\eeq

{\bf 2.} Given $\theta \in \Theta$, we now consider 
$$
\limsup_{n \to \infty} ~J_{\ve}\left( \min_{t \ge 0} | \chi_n(\theta,t) - x_n | \right) {\bf 1}_{\Theta_n^c(\lambda)} (\theta). 
$$

Here and in the sequel, ${\bf 1}_S$ denotes the characteristic function of a set $S$. If this limsup is strictly
positive, then there exists a subsequence of indices $n_k\to\infty$ and times $t_{n_k} \in [0,\lambda]$ such that 
\begin{itemize}
\item $\ds \theta \in \Theta_{n_k}^c(\lambda)$ for every $n_k$,
\item $\ds \min_{t \ge 0} | \chi_{n_k}(\theta,t) - x_{n_k} | = |\chi_{n_k} (\theta,t_{n_k}) - x_{n_k}|$,
\item $\ds \limsup_{n\to\infty} J_{\ve}\left( \min_{t\ge0} |\chi_n(\theta,t) - x_n| \right) {\bf 1}_{\Theta_n^c(\lambda)}(\theta) ~=~ \lim_{k\to\infty} J_{\ve}\Big(\bigl|\chi_{n_k}(\theta,t_{n_k}) - x_{n_k} \bigr|\Big) $.
\end{itemize}
By \eqref{pum2} it follows that $\theta \in \Theta^c(\lambda)$. Moreover, since $t_{n_k} \in [0,\lambda]$ for every $n_k$, by possibly selecting a further subsequence we obtain the convergence $t_{n_k} \to\bar{t} \in [0,\lambda]$. Since $d(\chi_n(\theta),\chi(\theta)) \to 0$, this implies the convergence $\chi_{n_k}(\theta) \to \chi(\theta)$ uniformly on $[0,\lambda]$. Therefore
$$
\chi_{n_k}(\theta,t_{n_k}) \to \chi(\theta,\bar{t}).
$$
Consequently
$$
\bigl| \chi_{n_k}(\theta,t_{n_k})  - x_{n_k}\bigr|\, \to\, \bigl| \chi(\theta,\bar{t}) - x\bigr|
\, \ge\, \min_{t \ge 0} \bigl| \chi(\theta,t) - x \bigr|. 
$$
Since $J_{\ve}$ is decreasing, this implies
$$
\lim_{k \to \infty} J_{\ve} \Big(\bigl| \chi_{n_k}(\theta,t_{n_k}) - x_{n_k}  \bigr|\Big)\, \le\, J_{\ve} \left( \min_{t \ge 0} | \chi(\theta,t) - x | \right). 
$$
This proves that 
\bel{pum3}
\limsup_{n\to\infty}\, J_{\ve}\left( \min_{t\ge0} |\chi_n(\theta,t) - x_n| \right) {\bf 1}_{\Theta_n^c(\lambda)}(\theta)
~ \le~ J_\ve\left( \min_{t \ge 0} | \chi(\theta,t) - x | \right) {\bf 1}_{\Theta^c(\lambda)}.
\eeq

{\bf 3.} By \eqref{pum1} and \eqref{pum3}, recalling that $J_{\ve} \le 1$ and using Fatou's lemma, we obtain
$$
\bega{rl}
\ds \limsup_{n\to\infty}  W_\ve(x,\chi_n) \!\! &=~ \ds \limsup_{n\to\infty}  \int_{\Theta} J_{\ve}\left(  \min_{t \ge 0} \bigl|\chi_n(\theta,t) - x_n
\bigr| \right) \, d\theta  \\[4mm]
&\le ~\ds \limsup_{n\to\infty}  \int_{\Theta_n^c(\lambda)} J_{\ve}\left(  \min_{t \ge 0} \bigl|\chi_n(\theta,t) - x_n\bigr| \right) \, d\theta +\delta \\[4mm]
&\le~\ds \int_{\Theta} \limsup_{n\to\infty}  \left\{ J_{\ve}\left(  \min_{t \ge 0} \bigl|\chi_n(\theta,t) - x_n\bigr| \right)  {\bf 1}_{\Theta_n^c(\lambda)}(\theta) \right\} d\theta + \delta \\[4mm]
&\le~\ds \int_{\Theta}  J_{\ve}\left(  \min_{t \ge 0} \bigl| \chi(\theta,t)  - x\bigr| \right) {\bf 1}_{\Theta^c(\lambda)}(\theta) \, d\theta + \delta \\[4mm]
&\ds\le W_\ve(x,\chi) + \delta .
\enda
$$
Since $\delta > 0$ is arbitrary, this concludes the proof. 
\endproof

The next result yields the lower semicontinuity of the mollified irrigation cost. 
It provides an analog of Proposition~3.40 in \cite{BCM}.

 \begin{proposition}\label{p:32}  Given any $C>0$,
let $\{ \chi_n \} \subset \IPC$ be a sequence of irrigation plans, 
all parameterized by arc-length as in (\ref{alp}).  Assume
$\chi_n \to \chi$ for some irrigation plan $\chi$. Then
$$ \E^{\alpha}_{\ve} (\chi)~\leq~
\liminf_{n \to \infty } \,\E^{\alpha}_{\ve}(\chi_n) .
$$ 
\end{proposition}

{\bf Proof.}
By Proposition \ref{p:31} and since the stopping times are lower semicontinuous, we have 
$$
\liminf_{n\to\infty}  \Big\{ \bigl[ W_\ve(\chi_n(\theta,t), \chi_n) \bigr]^{\alpha - 1}{\bf 1}_{[0,T_n(\theta)]}(t) \Big\}~ \ge
~ \bigl[ W_\ve(\chi(\theta,t),\chi) \bigr]^{\alpha - 1} {\bf 1}_{[0,T(\theta)]}(t)
$$
for a.e.~$\theta \in \Theta$ and all $t \in \R_+$. Therefore, by Fatou's lemma,
$$
\bega{rl}
\ds \liminf_{n\to\infty} ~ \E_{\ve}^{\alpha}(\chi_n) &\ds= ~\liminf_{n\to\infty} \int_{\Theta} \int_{\R_+} \bigl[ W_\ve(\chi_n(\theta,t),\chi_n) \bigr]^{\alpha - 1} \bigl|\dot{\chi}_n(\theta,t)\bigr| \, dt d\theta \\[4mm]
&\ds=~ \liminf_{n\to\infty}  \int_{\Theta} \int_{\R_+}  \bigl[ W_\ve(\chi_n(\theta,t),\chi_n) \bigr]^{\alpha - 1}{\bf 1}_{[0,T_n(\theta)]}(t) \, dt d\theta \\[4mm]
&\ds \ge~ \int_{\Theta} \int_{\R_+} \bigl[ W_\ve(\chi(\theta,t),\chi) \bigr]^{\alpha - 1} {\bf 1}_{[0,T(\theta)]}(t) \, dt d\theta \\[4mm]
&\ds \ge ~\int_{\Theta} \int_{\R_+} \bigl[ W_\ve(\chi(\theta,t),\chi) \bigr]^{\alpha - 1} |\dot{\chi}(\theta,t)| \, dt d\theta \\[4mm]
&\ds= ~\E_{\ve}^{\alpha}(\chi),
\enda
$$
because  $|\dot{\chi}(\theta,t)| \le 1$. This concludes the proof. 
\endproof

Assuming that the Radon measure $\mu$ admits an irrigation plan, 
thanks to the lower semicontinuity of the mollified cost, we can now prove the existence of a minimizer.

\begin{theorem}\label{t:existence}
Assume that the measure $\mu$ admits an irrigation plan with finite cost.
Then there exists an admissible irrigation plan $\chi$
for the measure $\mu$ that minimizes the mollified cost $\E^{\alpha}_\ve(\chi)$. 
\end{theorem}
{\bf Proof. } Since the mollified cost is not larger than the original cost, by assumption there exists a sequence of irrigation plans
$\chi_n \in\TP(\mu)$, ~$n\geq 1$, minimizing the mollified cost.  Without loss of generality, we can assume that each $\chi_n$ is parameterized by arc length as in (\ref{alp}). Since the sequence is minimizing, there exists a constant $C > 0$ such that
$$
\sup_{n \ge 1}~ \E^{\alpha}_\ve(\chi_n) ~\le~ C.
$$ 
Therefore,
$$
C ~\ge~ \sup_{n \ge 1} \int_{\Theta} \int_0^{T_{n}(\theta)} \left[ W_\ve(\chi_n(\theta,t), \chi_n) \right]^{\alpha - 1} |\dot{\chi}_n(\theta,t)| \, dt d\theta ~\ge~ \sup_{n \ge 1} \int_{\Theta} M^{\alpha - 1} T_n(\theta) \, d\theta,
$$
showing that the stopping times $T_n(\theta)$ satisfy a uniform bound of the form \eqref{um2}. 

By the compactness property stated in Theorem~\ref{t:compactness},  there exists $\chi \in \TP(\mu)$ such that (by possibly taking a subsequence and relabeling) $\chi_n \to \chi$. 
By Proposition \ref{p:32}, we conclude that $\chi$ is 
a minimizer of the mollified irrigation cost $\E^{\alpha}_\ve$. 
\endproof

\section{$\Gamma$-limit of the mollified irrigation costs}
\label{sec:4}
\setcounter{equation}{0}
%\subsection{Convergence of minimizers}\label{s:CM}
The mollified irrigation costs $\E_\ve^\alpha$ are of interest insofar as they yield an effective tool to approximate 
the original cost $\E^\alpha$. In this direction, a natural framework is provided by
 $\Gamma$-convergence. 
 In the following, we fix a constant $C>0$ and consider the family of all irrigation plans
 $\chi:\Theta\times\R_+\mapsto\R^d$, with $\chi\in \IPC$.
 In this setting, the $\Gamma$-convergence of the family of functionals $\E^{\alpha}_\ve$ to $\E^{\alpha}$ 
 is equivalent to the following conditions. 

For any sequence $( \ve_n)_{n\geq 1}$ such that $\ve_n \to 0$, one has:
\begin{itemize}

\item[${\bf (\Gamma 1)}$] For any sequence  of irrigation plans $\chi_n\in \IPC$, with $\chi_n \to\chi \in  \IPC$ as $n\to\infty$, there holds
\bel{G1}
\liminf_{n \to \infty} \,\E^{\alpha}_{\ve_n} (\chi_n) \,\ge\,\E^{\alpha}(\chi).
\eeq
\item[${\bf (\Gamma 2)}$] For every $\chi \in \IPC$, there exists a sequence of irrigation plans $\chi_n 
\in  \IPC$ such that $\chi_n \to \chi$ as $n\to\infty$ and 
$$
\limsup_{n \to \infty}\, \E^{\alpha}_{\ve_n}(\chi_n) \,\le\, \E^{\alpha}(\chi).  
$$
\end{itemize}

Toward a proof of $\Gamma$-convergence, we start by proving a simple relation between the mollified multiplicity $\eqref{mmi}$ and the original one. 

\begin{proposition}\label{p:41}
Let $\chi \in  \IPC$ be an irrigation plan and consider any point $x \in \R^d\setminus\{0\}$. Then there holds
\bel{i11}
W_\ve(x,\chi) \,\ge\, |x|_{\chi} 
\eeq
for all $\ve > 0$, and
\bel{i12}
\lim_{\ve \to 0} \,W_\ve(x,\chi) \,= \,|x|_{\chi}.
\eeq
Moreover, the mollified irrigation costs satisfy
\bel{i13}
\E^{\alpha}_{\ve}(\chi) \,\le\, \E^{\alpha}(\chi) \qquad \text{ and } \qquad
\lim_{\ve \to 0 }\, \E^{\alpha}_{\ve}(\chi) \,=\, \E^{\alpha}(\chi).  
\eeq
\end{proposition}
{\bf Proof.} {\bf 1.} The inequality \eqref{i11} was already proved at (\ref{i10}).

To prove (\ref{i12}), consider any $x\in\R^d$. For any $\theta \in \Theta$, define
$$
\rho(\theta,x)\, \doteq\, \min_{t \ge 0}\, \bigl|\chi(\theta,t) - x\bigr|.
$$
By (\ref{Jep}) this yields
$$
 W_\ve(x,\chi) ~=~\int_\Theta J_\ve\bigl(\rho(\theta,x)\bigr)\, d\theta~=~\int_\Theta J
\left({\rho(\theta,x)\over\ve}\right) d\theta.
$$
Taking the limit as $\ve\to 0$, by (\ref{Jprop}) it follows
$$\lim_{\ve \to 0+}J\left({\rho\over\ve}\right)  =~\left\{\bega{rl} 0 \quad &\hbox{if}~~\rho>0,\\[1mm]
 1 \quad &\hbox{if}~~\rho=0.\enda\right.$$
 Hence, by the dominated convergence theorem, 
$$
\lim_{\ve \to 0+} \, W_\ve(x,\chi)~=~\meas\bigl( \{  \theta\in\Theta \, ; ~ \rho(\theta)  = 0 \} 
\bigr)
~ = ~|x|_{\chi}
$$
and this proves \eqref{i12}.
\v
{\bf 2.} The remaining statements in \eqref{i13} are a straightforward consequence of \eqref{i11}-\eqref{i12}. 
\endproof
\v
Next, we prove a lower semicontinuity result similar to Proposition~\ref{p:31}, where the mollification parameter 
$\ve$ now converges to zero. 

\begin{proposition}\label{p:42} Consider a 
sequence of irrigation plans $\chi_n \in  \IPC$ and numbers $\ve_n>0$ such that 
$$
\chi_n \,\to\, \chi\in  \IPC, \qquad \ve_n \to 0 \qquad \text{ as } n \to \infty. 
$$
In addition, assume that the corresponding stopping times satisfy the boundedness property \eqref{um2}
with a uniform constant $C$. Then, the mollified multiplicities \eqref{mmi} satisfy
$$
\limsup_{n \to \infty} \,W_{\ve_n}\bigl( \chi_n(\theta,t) , \chi_n \bigr) \,\le\, \bigl|\chi(\theta,t)\bigr|_{\chi}
$$
for  all $t\in \R_+$ and a.e. $\theta \in \Theta$. 
\end{proposition}
{\bf Proof.} {\bf 1.} Let be $\theta_0 \in \Theta$ such that $d\bigl(\chi_n(\theta_0),\chi(\theta_0)\bigr) \to 0$. This convergence occurs for a.e. $\theta_0 \in \Theta$. For a fixed $t_0 \in \R_+$, setting
\bel{xnx}
x_n \doteq \chi_n(\theta_0,t_0) , \quad x \doteq \chi(\theta_0,t_0)
\eeq
we have the convergence 
$x_n \to x$. We  claim that 
\bel{cla1}
\limsup_{n \to \infty}\, W_{\ve_n}(x_n, \chi_n) \,\le\, |x|_{\chi}.
\eeq
\v
{\bf 2.} To prove (\ref{cla1}), consider the sets
\bel{Then}
\Theta_n(\lambda) \doteq\, \bigl\{\theta \in \Theta \, ; ~ T_n > \lambda\bigr\}, \qquad\qquad \Theta(\lambda) \,\doteq\, 
\bigl\{\theta \in \Theta \, ;~ T(\theta) > \lambda \bigr\},\eeq
\bel{xchi} [ x ]_{\chi} \,\doteq\, \bigl\{ \theta \in \Theta \, ;~ x = \chi(\theta,t)~~ \hbox{for some} ~t\geq 0\bigr\}. 
\eeq
As in Step {\bf 1} of the proof of Proposition~\ref{p:31}, by (\ref{liT}) it follows
$$
\bigcap_{k\geq 1} \bigcup_{n \ge k} \Theta_n^c(\lambda)~ \subseteq ~\Theta^c(\lambda).
$$
By (\ref{um2}), for every $\delta > 0$ and $\lambda > C/\delta$ there holds
$$
\meas\Big(\Theta_n(\lambda)\Big)\,\leq \,\delta\qquad  \text{ for every } ~n\geq 1.
$$
Moreover, by  the  convergence of the irrigation plans, 
$$\hbox{for a.e.}~
\bar \theta ~\in ~ \bigcap_{k\geq 1} \bigcup_{n \ge k} \Theta_n^c(\lambda)% \cap ([x]_{\chi})^c
$$
and for every subsequence $(n_k)_{k\geq 1}$ such that $\bar \theta  \in \Theta^c_{n_k}$ 
(so that $T_{n_k}(\bar \theta)\leq \lambda$)  for all $k\geq 1$, 
we have the convergence 
$$\chi_{n_k}(\bar \theta ,t)\, \to\, \chi(\bar \theta ,t)$$ uniformly for $t\geq 0$.

Next, assume that the point $x$ in (\ref{xnx}) satisfies
$$x \,\notin\,\hbox{Range}~\bigl( \chi(\bar \theta )\bigr)
~\doteq~\bigl\{\chi(\bar \theta, t)\,;~t\geq 0\bigr\}.$$ Since Range$\bigl( \chi_n(\bar \theta )\bigr)$,
Range$\bigl( \chi(\bar \theta )\bigr)$ are all compact sets and $x_n\to x$ as $n\to \infty$,
 there exists a radius $\bar r>0$ such that for every $n$ large enough
\bel{disj}
B(x_{n},\bar r) \cap \hbox{Range}\bigl(\chi_{n}(\bar \theta )\bigr) \,=\, \emptyset .
\eeq
This implies 
\bel{ineq}
\limsup_{n \to \infty} \,J_{\ve_n} \!\left( \min_{t \ge 0} | \chi_n(\bar \theta ,t) - x_n| \right)
% {\bf 1}_{\Theta_n^c(\lambda) \cap ([x]_{\chi})^c}(\bar \theta )
  \,=\, 0.  
\eeq
\v
{\bf 3.} Using (\ref{ineq}) we can now prove the desired inequality:
$$
\bega{rl}
\ds  \limsup_{n\to\infty} W_{\ve_n}(x_n,\chi_n) &\ds=~ \limsup_{n\to\infty} \int_{\Theta} J_{\ve_{n}}\!\left(  \min_{t \ge 0} |\chi_{n}(\theta,t) - x_{n}| \right) \, d\theta \\[4mm]
&\ds \le~ \limsup_{n\to\infty} \int_{[x]_{\chi}} J_{\ve_{n}}\!\left(  \min_{t \ge 0} |\chi_{n}(\theta,t) - x_{n}| \right) {\bf 1}_{\Theta_n^c(\lambda)}(\theta) \, d\theta \\[4mm]
&\ds \qquad + \limsup_{n\to\infty} \int_{\Theta\setminus [x]_{\chi}} J_{\ve_{n}}\!\left(  \min_{t \ge 0} |\chi_{n}(\theta,t) - x_{n}| \right) {\bf 1}_{\Theta_n^c(\lambda)}(\theta) \, d\theta + \delta \\[4mm]
&\le~ |x|_{\chi} + 0 +\delta.
\enda
$$
Since $\delta>0$ was arbitrary, this completes  the proof. 
\endproof

Combining the previous results one obtains
\begin{theorem} \label{t:43} 

On the family of all  irrigation plans $\chi\in \IPC$,
the  functionals $\E^{\alpha}_\ve$ defined at (\ref{mmi})-(\ref{TCep})  $\Gamma$-converge to $\E^{\alpha}$ as $\ve \to 0$. 
\end{theorem}
{\bf Proof.} The property ${ \bf ( \Gamma2 )}$ is an immediate consequence of Proposition \ref{p:41}, taking
$\chi_n=\chi$ for every $n\geq 1$.  On the other hand, the property ${\bf (\Gamma1)}$ is obtained using
Proposition~\ref{p:42} and repeating the same arguments as in the proof of Proposition~\ref{p:32}. 
\endproof

\section{An alternative mollification procedure}
\label{sec:5}
\setcounter{equation}{0}
In this section we consider a somewhat different mollification procedure, where the 
supremum in (\ref{mmi}) is replaced by an integral over time.   For computational purposes, this 
method appears to be easier to implement numerically. 
We will prove that the properties ${\bf (\Gamma1)}$ and ${\bf (\Gamma2)}$ remain valid for the corresponding 
mollified irrigation cost. 
%
%{\color{red}  ${\bf (\Gamma2)}$ is clear, but  ${\bf (\Gamma1)}$ is related to lower semicontinuity,
%which does not hold...}
%
 However, as it will be shown by a counterexample, this new irrigation cost is not lower semicontinuous.
We thus cannot guarantee the existence of global minimizers.

To fix ideas, 
consider a smooth, positive, decreasing function $J:\R_+\mapsto\R_+$ such that
\bel{Jpro2}
1~\leq~\int_0^{+\infty} J(r)\, dr ~<~+\infty,\eeq
and set $J_\ve(r) = J(r/\ve)$, as in (\ref{Jep}).
Observe that the monotonicity and the integrability assumptions together imply
\bel{Jpro3}\lim_{r\to+\infty} r\,J(r)~=~0.\eeq
Otherwise we could find an increasing sequence $(r_k)_{k\geq 0}$ with $r_k> 2 r_{k-1}$, such that 
$r_k J(r_k)\geq \delta>0$ for every $k$.   This would imply
$$\int_0^{+\infty} J(r)\, dr ~\geq~\sum_{k\geq 1} \int_{r_k/2}^{r_k} J(r)\, dr~\geq~\sum_{k\geq 1} 
{r_k\over 2} J(r_k)~\geq~\sum_{k\geq 1} {\delta\over 2} ~=~+\infty.$$

Given an irrigation plan $\chi\in\TP(\mu)$, we now define the mollified multiplicity of a point $x\in\R^d$ as
\bel{mm1}\Tilde W_\ve(x,\chi)
~\doteq~\int_\Theta\min\left\{ 1, \int_0^{+\infty} {1\over \ve} \, J_\ve\Big(\bigl| \chi(\theta,t)-x\bigr| \Big)\cdot |\dot \chi(\theta,t)|\
\, dt \right\}d\theta.\eeq
The corresponding mollified irrigation cost is then defined as
\bel{mic1}\Tilde \E^\alpha_\ve(\chi)~\doteq~\int_\Theta \left(\int_{\R_+} \bigl[ \Tilde W_\ve(\chi(\theta,t),\chi)
\bigr]^{\alpha-1} \cdot |\dot \chi(\theta,t)|\, dt\right)
d\theta.\eeq
For the family of mollified irrigation costs $\Tilde \E^\alpha_\ve(\chi)$,
the property  ${\bf (\Gamma2)}$ is an immediate consequence of

\begin{proposition}\label{p:61}
Let $\chi \in \IPC$ be an admissible irrigation plan and consider any point $x \in \R^d\setminus\{0\}$. Then there holds
\bel{i61}
\lim_{\ve \to 0} \,\Tilde W_\ve(x,\chi) \,= \,|x|_{\chi}.
\eeq
Moreover, the mollified irrigation costs satisfy
\bel{i62}
\lim_{\ve \to 0 }\, \Tilde\E^{\alpha}_{\ve}(\chi) \,=\, \E^{\alpha}(\chi).  
\eeq
\end{proposition}

{\bf Proof.} Without loss of generality, we can assume that the paths $t\mapsto \chi(\theta,t)$ are all parameterized by arc-length, as in (\ref{alp}).
Consider any point $x\in \R^d\setminus\{0\}$.  For any given $\theta\in \Theta$, two cases may occur.

CASE 1: $x=\chi(\theta, \tau)$ for some $\tau>0$.   In this case
$$\int_0^{T(\theta)} {1\over \ve} \, J_\ve\Big(\bigl| \chi(\theta,t)-x\bigr| \Big)
\, dt~\geq~\int_0^\tau {1\over \ve} \, J_\ve(\tau-t)
\, dt~ ~=~\int_0^{\tau/\ve} J(\tau-t)\, dt.$$
Letting $\ve\to 0$, by  (\ref{Jpro2}) we obtain 
\bel{limi6}
\liminf_{\ve\to 0}\int_0^{T(\theta)} {1\over \ve} \, J_\ve\Big(\bigl| \chi(\theta,t)-x\bigr| \Big)
\, dt~\geq ~\int_0^{+\infty} J(r)\, dr~\geq ~1.\eeq

CASE 2: $x\not= \chi(\theta, t)$ for any $t\geq 0$.   
By compactness, this implies
$$\bigl| x-\chi(\theta, t)\bigr|~\geq~\delta~>~0$$
for all $t\geq 0$. Therefore
$$\int_0^{T(\theta)} {1\over \ve} \, J_\ve\Big(\bigl| \chi(\theta,t)-x\bigr| \Big)
\, dt~\leq~\int_0^{T(\theta)} {1\over \ve} \, J_\ve(\delta)~=~T(\theta) {J(\delta/\ve)\over\ve}\,.$$
Letting $\ve\to 0$, by  (\ref{Jpro3}) we obtain 
\bel{limi7}
\limsup_{\ve\to 0}\int_0^{T(\theta)} {1\over \ve} \, J_\ve\Big(\bigl| \chi(\theta,t)-x\bigr| \Big)
\, dt~\leq ~\lim_{\ve\to 0} ~T(\theta) {J(\delta/\ve)\over\ve}~=~0.\eeq
Together, (\ref{limi6}) and (\ref{limi7}) yield (\ref{i61}).

The limit (\ref{i62}) now follows by the dominated convergence theorem.
\endproof

The above proposition justifies the use of the above mollification approach as a computational tool.
On the other hand, the following counterexample shows that, for a fixed $\ve > 0$ and $0\leq \alpha <1$, the mollified irrigation cost $\Tilde \E^\alpha_\ve$
is not lower semicontinuous.   For this reason, the existence of a minimizer of the mollified cost
cannot be guaranteed.

\begin{example} {\rm  Let $\mu$ be a measure consisting 
of two point masses, say of size $m_1, m_2$, located at points
$P_1, P_2$.
As shown in Fig.~\ref{f:ir175}, let $\chi$ be an irrigation plan where all water particles move along
two disjoint paths $\gamma_1, \gamma_2$ of length $\ell_1, \ell_2$ respectively.
For $n\geq 1$, let $\chi_n$ be an irrigation plan where particles again move along
 two paths.  The first is always the same: $\gamma_{1,n} = \gamma_1$. The second path
 is  $\gamma_{2,n}$, obtained replacing a section of $\gamma_1$ with a zig-zag polygonal,
 so that its total length becomes larger. More precisely, we assume
 \bel{ellsize}
 \ell_2\,<\,\ell_{2,n}\,=\, \ell_2^+ \,<\!<\,1,\qquad\qquad \ell_1\,>\,1. \eeq As usual, all paths will be parameterized by arc-length.

We now assume that $\ve>0$ is sufficiently large and that all paths are sufficiently close to each other, so that
$${1\over\ve} J_\ve\bigl(|x_1-x_2|\bigr)~=~1$$
for every couple of points $x_1,x_2$ on any two of these curves.

In view of (\ref{ellsize}), the  mollified cost of the irrigation plan $\chi$ is given by
\bel{TEx}\bega{rl}
\Tilde\E(\chi)&\ds = ~\int_0^{\ell_1} \left[ m_1 +\int_0^{\ell_2} {1\over \ve} J_{\ve}\bigl( | \gamma_2(s) - \gamma_1(t) 
| \bigr)m_2 ds    \right]^{\alpha - 1} \,m_1 dt\\[4mm] 
&\ds \qquad + \int_0^{\ell_2} \left[m_1 +\int_0^{\ell_2} {1\over \ve} J_{\ve}\bigl( | \gamma_2(s) - \gamma_2(t) 
| \bigr) m_2ds    \right]^{\alpha - 1} \, m_2 dt\\[4mm]
&=~\ds [m_1+m_2\ell_2]^{\alpha-1} (m_1\ell_1+m_2\ell_2).
\enda
\eeq
A similar computation shows that the mollified cost of the irrigation plans $\chi_n$ is given by
$$
\Tilde\E(\chi_n)~=~  [m_1+m_2\ell_2^+]^{\alpha-1} (m_1\ell_1+m_2\ell_2^+).$$
Differentiating the right hand side of (\ref{TEx}) w.r.t.~$\ell_2$ one obtains
\bel{pel2}
\bega{l}\ds{\partial\over \partial \ell_2}  \Big( [m_1+m_2\ell_2]^{\alpha-1} (m_1\ell_1+m_2\ell_2)\Big)
\\[2mm]\qquad 
= ~  [m_1+m_2\ell_2]^{\alpha-1} m_2 +(\alpha-1)  [m_1+m_2\ell_2]^{\alpha-2}m_2(m_1\ell_1+m_2\ell_2)
\\[2mm]
\qquad \ds=~ [m_1+m_2\ell_2]^{\alpha-1} m_2\left( 1 - {(1-\alpha) (m_1\ell_1+m_2\ell_2) \over m_1+m_2\ell_2}\right).
\enda\eeq
We observe that, by choosing the length $\ell_1$ sufficiently large, the right hand side of (\ref{pel2})
becomes negative. Therefore, if $\ell_2^+$ is slightly greater than $\ell_2$, it follows
$$\Tilde\E(\chi_n)~<~\E(\chi),$$
showing that lower semicontinuity does not hold. 

\begin{figure}[ht]
\centerline{\hbox{\includegraphics[width=7cm]
{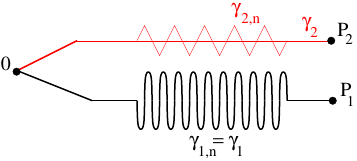}}}
\caption{\small An example showing that, when the mollified multiplicity \eqref{mm1} is adopted, the mollified irrigation cost is not lower semicontinuous. In the limit as $n\to\infty$,  the paths $\gamma_{2,n}$ are replaced by the shorter path $\gamma_2$. Hence the transportation cost along $\gamma_2$ decreases. However, along $\gamma_1$ the mollified multiplicity decreases and hence the transportation cost is larger.
}
\label{f:ir175}
\end{figure}
}
\end{example}

In spite of the previous example, which occurs for a fixed value of $\ve>0$, the lower semicontinuity result
stated in Proposition~\ref{p:42}  remains valid in the limit as $\ve_n\to 0$ also for the 
mollified multiplicity (\ref{mm1}).
 
\begin{proposition}\label{p:53} Consider a 
sequence of irrigation plans $\chi_n \in \IPC$ and numbers $\ve_n>0$ such that 
$$
\chi_n \,\to\, \chi\in \IPC , \qquad \ve_n \to 0 \qquad \text{ as } n \to \infty. 
$$
In addition, assume that the corresponding stopping times satisfy the boundedness property \eqref{um2}
with a uniform constant $C$. Then, the mollified multiplicities \eqref{mm1} satisfy
$$
\limsup_{n \to \infty} \,\Tilde W_{\ve_n}\bigl( \chi_n(\theta,t) , \chi_n \bigr) \,\le\, \bigl|\chi(\theta,t)\bigr|_{\chi}
$$
for a.e. $\theta \in \Theta$, $t\in \R_+$\,.
\end{proposition}

{\bf Proof.} We follow the same steps  as in the proof of Proposition~\ref{p:42}.

{\bf 1.} Let be $\theta_0 \in \Theta$ such that $d(\chi_n(\theta_0),\chi(\theta_0)) \to 0$. This convergence occurs for a.e. $\theta_0 \in \Theta$.   For a fixed $t_0 \in \R_+$, defining
$$
x_n \doteq \chi_n(\theta_0,t_0) , \quad x \doteq \chi(\theta_0,t_0),
$$
we have the convergence
$x_n \to x$. We claim that 
\bel{cla5}
\limsup_{n \to \infty}\, \Tilde W_{\ve_n}(x_n, \chi_n) \,\le\, |x|_{\chi}\,.
\eeq
The proof of (\ref{cla5}) is achieved by the same arguments used in step {\bf 2} of the proof
of Proposition~\ref{p:42}.    The only difference is that now the limit (\ref{ineq}) is replaced by
\bel{ineq5}
\limsup_{n \to \infty} 
\int_{\Theta\setminus [x]_\chi}\min\left\{ 1, \int_0^{+\infty} {1\over \ve_n} \, J_{\ve_n}\Big(\bigl| \chi_n(\bar \theta,t)-x\bigr| \Big)\cdot |\dot \chi_n(\bar \theta,t)|\
\, dt \right\}d\bar \theta
 \,=\, 0.  
\eeq
\v
{\bf 2.} Using (\ref{ineq5}) we can now prove the desired inequality:
$$
\bega{l} \ds\limsup_{n\to\infty}\,
 \Tilde W_{\ve_n}(x_n,\chi_n)\\[4mm] 
 \ds =~ \limsup_{n\to\infty} \int_{\Theta} \min\left\{ 1, \int_0^{+\infty} {1\over \ve_n} \, J_{\ve_n}\Big(\bigl| \chi_n(\theta,t)-x_n\bigr| \Big)\cdot \bigl|\dot \chi_n(\theta,t)\bigr|
\, dt \right\}d\theta  \\[4mm]
\ds \le~ \limsup_{n\to\infty} \int_{[x]_{\chi}}  \min\left\{ 1, \int_0^{+\infty} {1\over \ve_n} \, J_{\ve_n}\Big(\bigl| \chi_n(\theta,t)-x_n\bigr| \Big)\cdot \bigl|\dot \chi_n(\theta,t)\bigr|
\, dt \right\}{\bf 1}_{\Theta_n^c(\lambda)}(\theta) \, d\theta \\[4mm]
\ds \qquad + \limsup_{n\to\infty} \int_{\Theta\setminus [x]_{\chi}} \min\left\{ 1, \int_0^{+\infty} {1\over \ve_n} \, J_{\ve_n}\Big(\bigl| \chi_n(\theta,t)-x_n\bigr| \Big)\cdot \bigl|\dot \chi_n(\theta,t)\bigr|
\, dt \right\}{\bf 1}_{\Theta_n^c(\lambda)}(\theta) \, d\theta + \delta \\[4mm]
\le~ |x|_{\chi} +0 + \delta.
\enda
$$
Since $\delta>0$ is arbitrary, this achieves the proof. 
\endproof

{}From Proposition~\ref{p:53} it follows that the property ${\bf (\Gamma1)}$ is also satisfied. Summarizing
the previous analysis we thus have

\begin{theorem} \label{t:53} Let $\mu$ be a positive Radon measure on $\R^d$, with bounded support.
 On the family of  irrigation plans $\chi\in \IPC$,
the  functionals $\Tilde \E^{\alpha}_\ve$ defined at (\ref{mm1})-(\ref{mic1})  $\Gamma$-converge to $\E^{\alpha}$ as $\ve \to 0$. 
\end{theorem}

\section{Numerical simulations of optimal irrigation patterns}
\label{sec:6}
\setcounter{equation}{0}

Our first simulations illustrate the effect of the mollification parameter $\ve$ on the local
minimizer of the irrigation problem.

In Fig.~\ref{f:25-0,4} we consider the irrigation of 25 equal masses uniformly distributed along 
a half-circumference. The exponent in the irrigation cost (\ref{TCg}) is here $\alpha=0.4$.
Minimizers of the mollified cost functional (\ref{mic1}) are obtained by a local gradient descent algorithm, taking $\ve =  0.25$, $\ve=0.1$ and $\ve= 0.05$, respectively.    As the mollification parameter $\ve$ decreases, it is observed that the branches merge together, approaching the limit configuration shown in the right-most figure.   

Similarly, 
in Fig.~\ref{f:29-0,9} we consider the irrigation of 29 equal masses uniformly distributed along 
a half-circumference. The exponent in the irrigation cost (\ref{TCg}) is here $\alpha=0.9$.
Minimizers of the mollified cost functional (\ref{mic1}) are obtained by a local gradient descent algorithm, taking $\ve =  0.05$, $\ve=0.025$ and $\ve= 0.01$, respectively.    Again, as $\ve\to 0+$, the branches merge together, approaching the  limit configuration shown in the right-most figure.

This behavior can be understood by observing that, in a ramified transport, if $\alpha<1$
then the transportation cost gets discounted when water particles move along the same path.   In connection with the mollified functional (\ref{TCep}),   one still gets a discount 
as long as water particles remain a distance $<\ve$ from each other.     As $\ve\to 0$, the mollified irrigation cost converges to the original one, as stated in Theorem~\ref{t:53}.

\begin{figure}
\centerline{\hbox{\includegraphics[width=4cm]{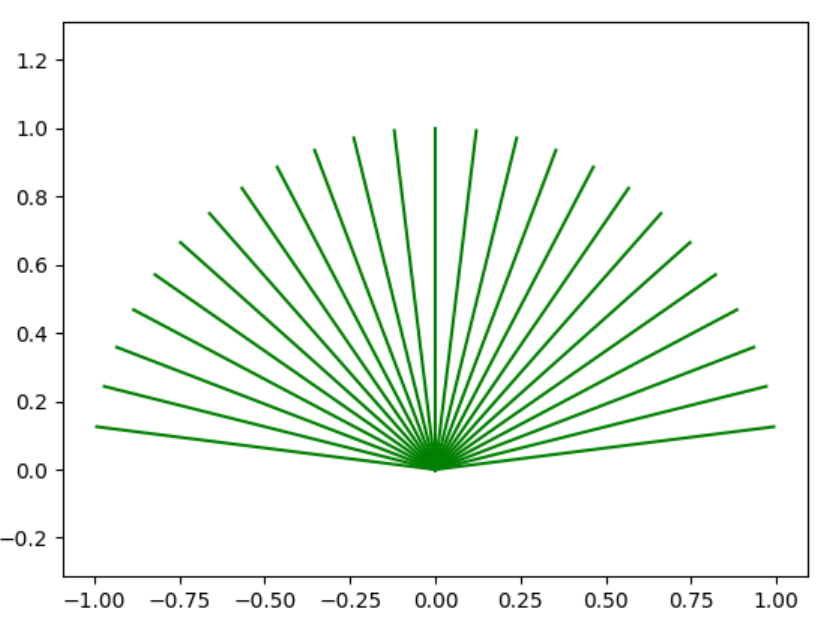}}
\hbox{\includegraphics[width=4cm]{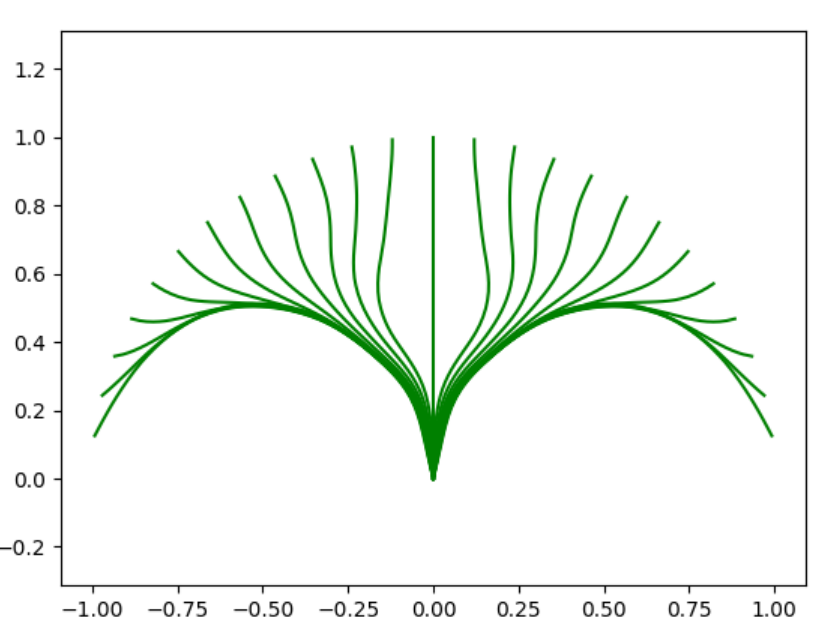}}
\hbox{\includegraphics[width=4cm]{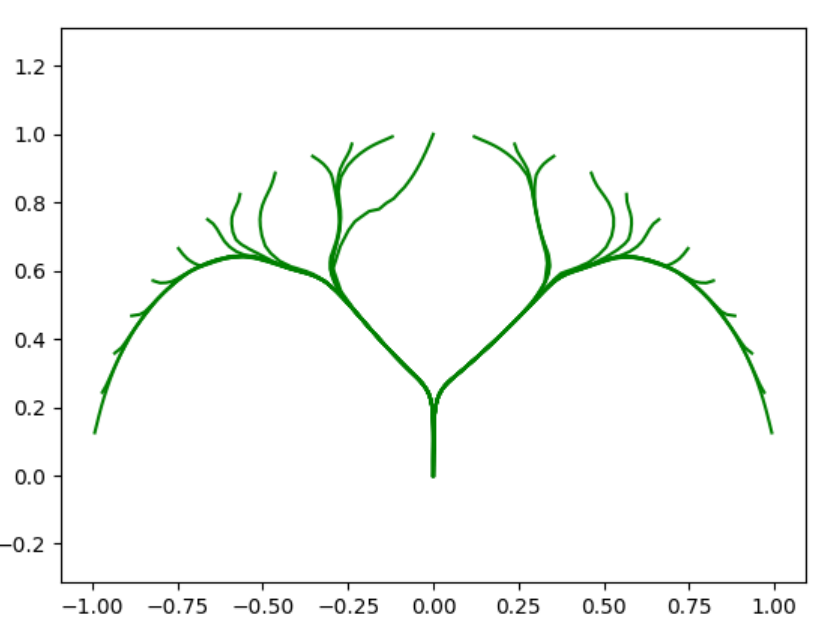}}
\hbox{\includegraphics[width=4cm]{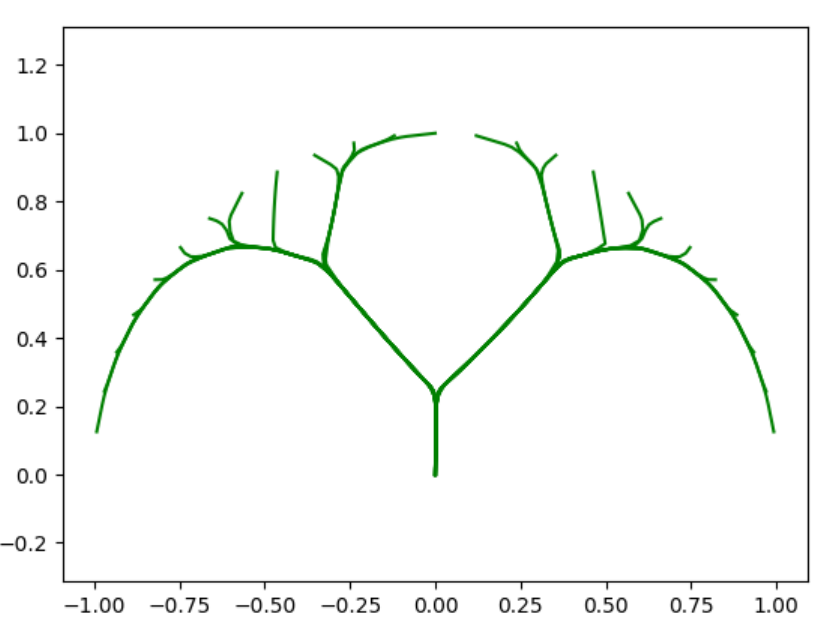}}
}
\caption{\small Different stages of the minimization of (\ref{mic1}), applied to the irrigation of 25 equal masses 
located along an arc of circumference.  Here $\alpha= 0.4$.    The left image represents the initial configuration. The remaining three figures represent the local minimizers obtained by 
gradient descent, where the parameter in the mollifier takes the values
$\ve =  0.25, 0.1, 0.05$, respectively.
}
\label{f:25-0,4}
\end{figure}

\begin{figure}
\centerline{\hbox{\includegraphics[width=4cm]{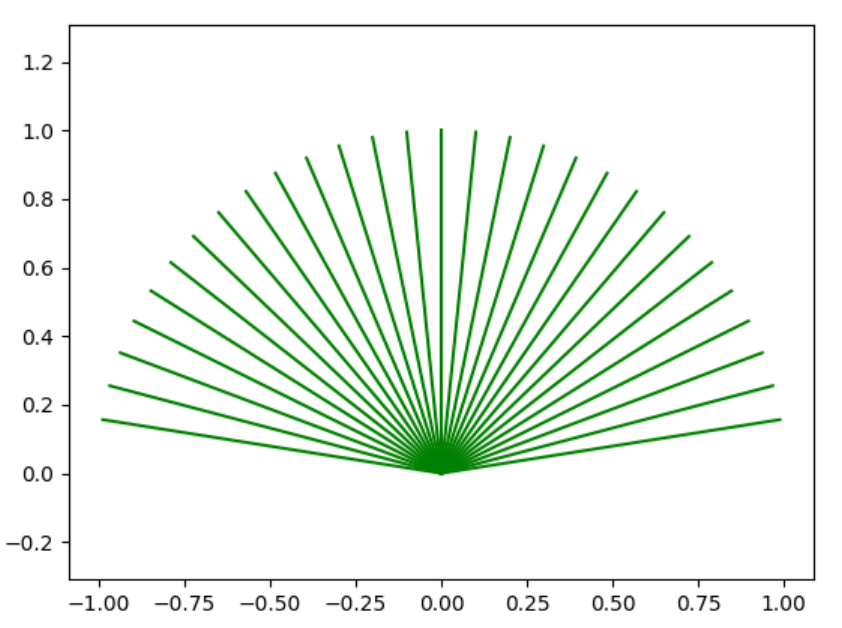}}
\hbox{\includegraphics[width=4cm]{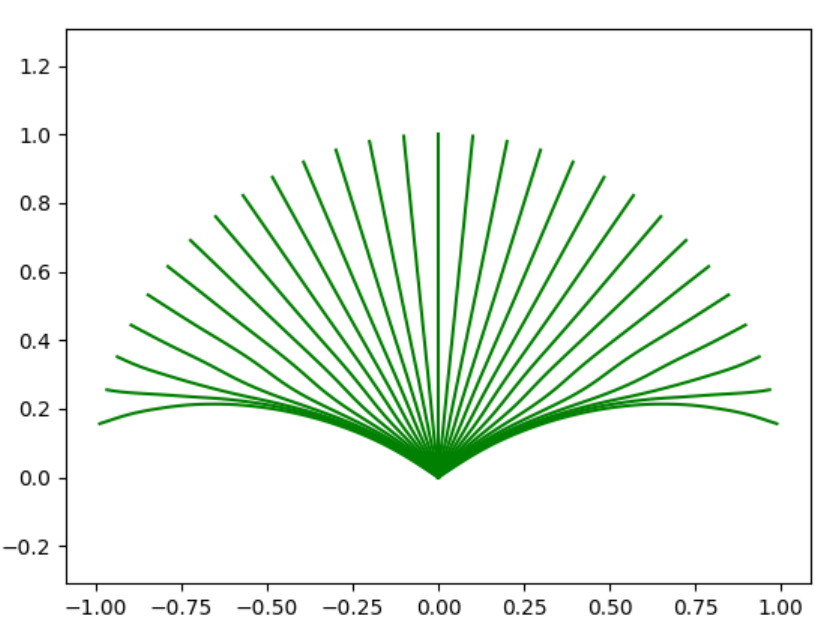}}
\hbox{\includegraphics[width=4cm]{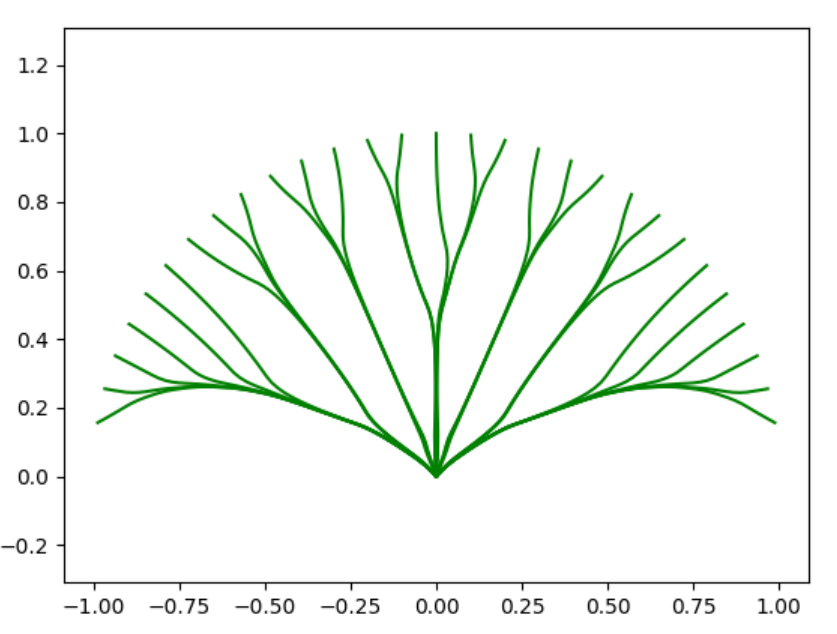}}
\hbox{\includegraphics[width=4cm]{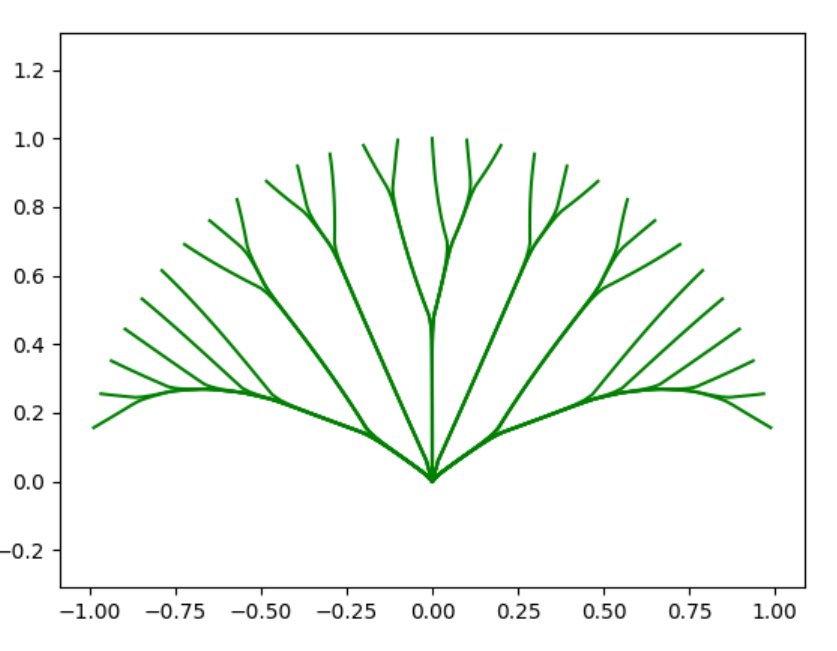}}
}
\caption{\small  Different stages of the minimization of (\ref{mic1}),  applied to the irrigation of 29 equal masses 
located along an arc of circumference.  Here $\alpha= 0.9$.   The left image represents the initial configuration. The remaining three figures represent the local minimizers obtained by 
gradient descent, taking $\ve= 0.1, 0.05, 0.01$, respectively.
  }
\label{f:29-0,9}
\end{figure}

\section{Optimal shapes of tree roots and branches}
\label{sec:7}
\setcounter{equation}{0}
The models of tree roots or tree branches in \cite{BSun}  are formulated in terms of a measure $\mu$
describing the density of root hair or leaves, respectively. 
This has to be optimized subject to an irrigation cost.  More precisely, in case of tree roots, 
the density $u(x)$ of water and nutrients in the soil at a point $x$ is determined 
as the solution to the elliptic equation
\bel{elleq}\Delta u+ f(u) - u\,\mu~=~0,\eeq
while 
\bel{Hdef}\H(\mu)~\doteq~\int u(x)\, d\mu(x)\eeq
is the total amount of nutrients harvested by the roots. 
Given a constant $c>0$, 
one seeks a measure $\mu$ which minimizes the combined functional
\bel{mumin}
\I^\alpha(\mu) - c\H(\mu).\eeq
Roughly speaking, to maximize the harvest functional $\H(\mu)$, the measure $\mu$ should 
not only have a large total mass, but also be spread out far and wide.   Otherwise the corresponding 
density $u(x)$ in (\ref{elleq})  would be small.
On the other hand, this requires a larger irrigation cost.   
The optimal shape provides a balance between 
these two conflicting goals.

Similar considerations apply to the functional $\S(\mu)$ in \cite{BSun},
describing how much sunlight is captured by a distribution
$\mu$ of leaves. This will be large if the leaves are spread out over a large surface.
However, this in turn increases the ramified transport cost.

In the present paper, to simplify the numerical simulations, 
we consider a payoff functional of the form
\bel{Hpen}c_2\Tilde \H(\mu)- c_1\P(\mu)~\doteq~c_2 \mu(\R^2) -c_1\dint K(x,y) \,d\mu(x)\,d\mu(y),\eeq
Here $\Tilde \H(\mu)$ is a payoff, measured simply by the total mass of the measure $\mu$.
This accounts for the total amount of leaves, or the total amount of root hair, dependign on the model.
On the other hand, the integral term $\P(\mu)$ penalizes configurations where a large amount of mass 
is concentrated on a small region.   In particular, one can choose 
a kernel $K:\R^2\mapsto\R$ of the form
\bel{K1}
 K(x,y)\,=\, \exp\bigl\{ -\beta\, |x-y|^2\bigr\}\eeq
 or  
\bel{K2}
K(x,y)\,=\, \bigl(1+|x-y|\bigr)^{-\gamma},\eeq
 with $\beta,\gamma>0$.
 This leads to the optimization problem:
 \begi
 \item[{\bf (OPT)}]  {\it Given $c_1, c_2>0$, find a positive Radon measure $\mu$ on $\R^2$ that minimizes
  the %combined
   functional}
 \bel{Jdef}\J(\mu)~=~\I^\alpha(\mu) + c_1 \P(\mu) - c_2 \Tilde \H(\mu).\eeq
 \endi
 The existence of a solution
 can be proved by the same arguments developed in \cite{BPS, BSun}, based  on lower semicontinuity 
 together with  a priori 
 bounds on the support of the measure $\mu$.     To state a precise result, we assume
 \begi
 \item[{\bf (A1)}]  {\it One has $K(x,y)=\phi\bigl(|x-y|\bigr)$ where $\phi:\R_+\mapsto\R_+$ is a 
 continuous, positive, 
 nonincreasing  function, with $\phi(0)>0$.}\endi

\begin{theorem} Let the kernel $K$ satisfy {\bf (A1)}.  If $\alpha> 1-{1\over d}$, then for any $c_1, c_2>0$ the 
problem {\bf (OPT)} has an optimal solution.
\end{theorem}

\begin{figure}[ht]
\centerline{\hbox{\includegraphics[width=8cm]
{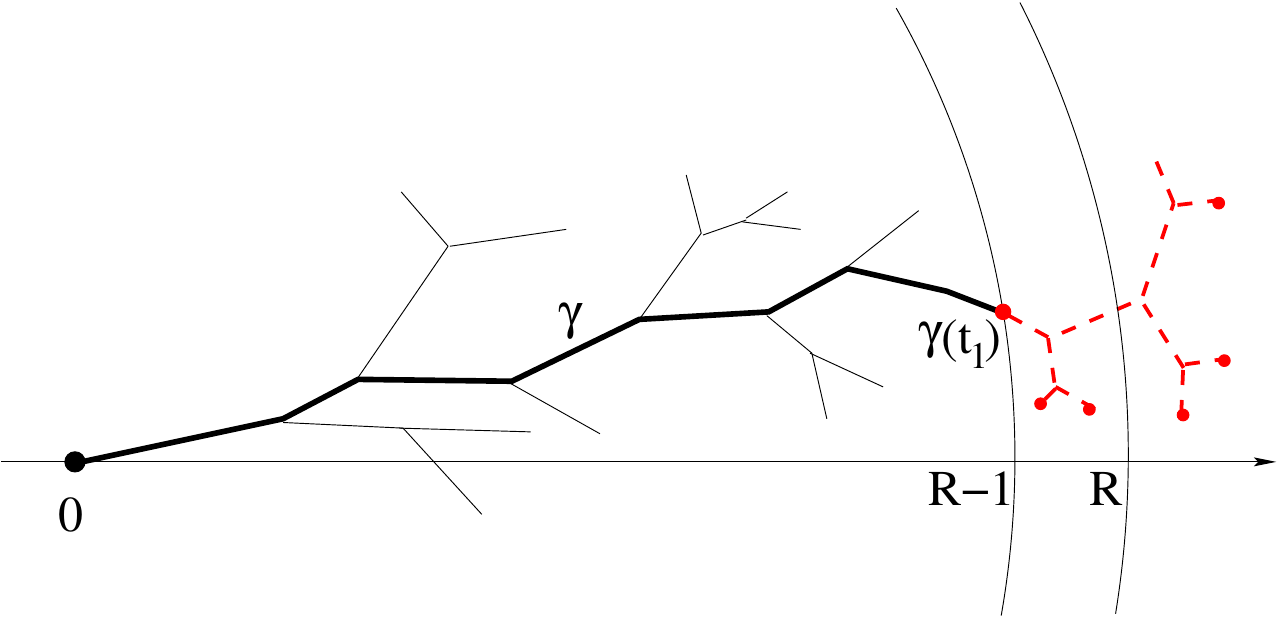}}}
\caption{\small   If the radius $R$ is sufficiently large, the cost of irrigating points outside the 
ball $B(0, R-1)$ is large.   By removing from the measure $\mu$ all masses along the dotted lines
beyond $\gamma(t_1)$, the total cost $\J(\mu)$ decreases.
}
\label{f:ir180}
\end{figure}

{\bf Proof.} {\bf 1.} Let $(\mu_n)_{n\geq 1}$ be a minimizing sequence of measures. 
By the assumption {\bf (A1)} there exist $r_0, K_0>0$ such that
 \bel{phibig}
\phi(r)\,\geq\,K_0\qquad\qquad \forall r\in \,]0, r_0]\,.\eeq
We  claim that without loss of generality one can assume that 
\bel{msmall}\mu_n\bigl( B(x,r_0/2)\bigr)~\leq~m_0~\doteq~{c_2\over c_1 K_0}\eeq
for every $n\geq 1$ and every ball of radius $r_0/2$.

Indeed, assume that
$\mu_n\bigl( B(x_0,r_0/2)\bigr)>m_0$
for some $x_0\in \R^d$.
Let $\Tilde \mu_n$ be the measure obtained from $\mu$ by removing all the mass contained inside the
ball $B(x_0, r_0/2)$.
We have
$$\I^\alpha(\Tilde \mu_n)~<~\I^\alpha(\mu_n),\qquad\quad 
 \mu_n(\R^2)-\Tilde\mu_n(\R^2) ~=~
\mu_n\bigl( B(x_0,r_0/2)\bigr),$$
$$  \P( \mu_n) - \P(\Tilde \mu_n)~\geq~\dint_{x,y\in B(x_0, r_0/2)}
K_0 \,d\mu_n(x) \,d\mu_n(y)~\geq~K_0 \Big[\mu_n\bigl( B(x_0,r_0/2)\bigr)\Big]^2.$$
If (\ref{msmall}) fails, then the above relations imply
$$\J(\Tilde\mu_n)-\J(\mu_n)~<~-c_1\mu_n\bigl( B(x,r_0/2)\bigr) +c_2 K_0 \Big[\mu_n\bigl( B(x_0,r_0/2)\bigr)\Big]^2~\leq ~0.$$
Hence we can replace $\mu_n$ with $\Tilde\mu_n$, decreasing the total cost.
\v
{\bf 2.} Next, observe that in $\R^d$ the ball $B(0, R)$ can be covered by finitely many  balls of radius $r_0/2$.  Indeed, one can achieve a covering with $\leq C' (1+R^d)$ such balls, for some constant $C'$ depending only on the space dimension.  As a consequence, 
by (\ref{msmall}) we can assume that 
our minimizing sequence satisfies the additional condition
\bel{mnr}\mu_n\bigl(B(0, R)\bigr)~\leq~C_0 (1+R^d)\eeq
for some constant $C_0$ and all $R>0$ and $n\geq 1$.  
Notice that by (\ref{msmall}) the measures $\mu_n$ can contain point masses, but only
with magnitude $\leq m_0$.
\v
{\bf 3.} In this step we seek an a priori bound on the support of the measures $\mu_n$.   
Given a measure $\mu$ with finite irrigation cost $\I^\alpha(\mu)<+\infty$,
let $R$ be a radius such that
\bel{71}\mu\bigl\{ x\in\R^d\,;~~|x|\geq R\bigr\} \,>\,0,
\qquad\qquad \mu\bigl\{ x\in\R^d\,;~~|x|> R\bigr\} \,\leq\,m_0\,.\eeq
Let $\chi$ be an optimal irrigation plan for $\mu$.

As shown in Fig.~\ref{f:ir180}, 
let $t\mapsto \gamma(t)$, $t\in [0, t_1]$, be an arc-length parameterization of a branch 
transporting a positive amount of mass, such that
\bel{75}\gamma(0)=0,\qquad \big|\gamma(t_1)\big|~=~R-1,\eeq
and such that the total flux across $\gamma$ satisfies
$$\bigl|\gamma(t)\bigr|_{\chi}~\geq~\bigl|\gamma(t_1)\bigr|_{\chi} ~\doteq~m_1~>~0
\qquad\quad\forall ~t\in [0,t_1].$$
Note that by the first inequality in (\ref{71}) such a branch exists.   Moreover,
if $\mu$ satisfies the condition in  (\ref{mnr}), then 
\bel{72} \mu(\R^d)\leq C_0 (1+R^d) + m_0\,,\qquad\qquad t_1\geq R-1.\eeq
Here the second inequality follows from (\ref{75}) 
and the fact that branches are parameterized by arc-length.
\v
{\bf 4.}
 Let $\Tilde \mu$ be the measure obtained from $\mu$ by removing all the particles
 which are transported along $\gamma$ up to the point $\gamma(t_1)$ or beyond. Call
 $$m^*~\doteq~\mu(\R^d)-\Tilde\mu(\R^d)$$
 the mass of the removed particles (i.e., those traveling along the dotted lines in Fig.~\ref{f:ir180}).
 
 Recalling the assumption $1+(\alpha-1)d>0$
the difference in the irrigation costs is estimated as
\bel{73}\bega{rl}\ds \I^\alpha(\Tilde \mu)&\ds \leq~\I^\alpha(\mu) - \int_0^{t_1} 
\Big(\bigl|\gamma(t)\bigr|_{\chi}^\alpha\, dt + \int_0^{t_1} \bigl( \bigl|\gamma(t)\bigr|_{\chi} - m^*\bigr)^\alpha  \,dt\\[4mm]
 &\ds \leq~\I^\alpha(\mu) - \int_0^{t_1} 
\bigl|\gamma(t)\bigr|_{\chi}^{\alpha-1} m^*  \,dt\\[4mm]
&\ds \leq~\I^\alpha(\mu) - \int_0^{t_1} 
\bigl( C_0 (1+R^d) +m_0\bigr)^{\alpha-1} m^*  \,dt\\[4mm]
&\ds \leq~\I^\alpha(\mu) - (R-1)
\bigl( C_0 R^d +C_0+m_0\bigr)^{\alpha-1} m^*\\[4mm]
&\ds \leq~\I^\alpha(\mu) - \delta_0\, R^{1+(\alpha-1) d}m^*,\\[4mm]
\enda \eeq
for some constant $\delta_0>0$ and all $R>\ov R$, for some $\ov R$ sufficiently large, depending only on $C_0$ and $m_0$.

Since trivially $\P(\Tilde \mu)\leq \P(\mu)$, we conclude that the total cost
of the reduced measure satisfies
$$\J(\Tilde\mu)-\J(\mu)~\leq~\bigl[\I^\alpha(\Tilde\mu) - \I^\alpha(\mu) \bigr]- \bigl[\Tilde\mu(\R^d)-\mu(\R^d)
\bigr]~\leq~- \delta_0\, R^{1+(\alpha-1) d}m^*  + m^*~<~0$$
as soon as $R>\ov R$.
By possibly repeating several times the above construction, we obtain a new measure $\Hat\mu$,
whose support is contained in the ball  $B(0,\ov R)$, and has smaller cost: $\J(\Hat\mu) <\J(\mu)$.
\v
{\bf 5.}
W.l.o.g.~we can thus assume that all measures $\mu_n$ of our minimizing sequence
are supported inside a fixed ball $B(0,\ov R)$.   By (\ref{mnr}) their total mass is uniformly bounded.
In turn, this implies that the irrigation costs $\I^\alpha(\mu_n)$ are uniformly bounded. In particular, 
recalling the definition at (\ref{stopC}), we conclude that
all the  corresponding optimal irrigation plans satisfy $\chi_n\in \IPC$ for some uniform constant $C$.

The existence of an optimal measure now follows
from the compactness Theorem~\ref{t:compactness} and the lower semicontinuity of the functionals 
$\I^\alpha(\mu)$ and $\P(\mu)$.
\endproof

\section{Numerical simulations of optimal tree shapes}
\label{sec:8}
\setcounter{equation}{0}

In the numerical tests in this section we seek the optimal location of $n$ tree branches, and the density of leaves on each one of them. The unknowns are 
piece-wise linear function $x_k(s)$, $y_k(s)$, and a piece-wise constant leaf density $m_k(s)$ for $k = 1, 2,\ldots, n$.
We aim to find the optimal values of these functions by minimizing an approximation to the cost functional 
\begin{equation}\label{cost}
\J(x_k, y_k, m_k)=\I^\alpha_\ve(x_k, y_k, m_k) + c_1 \P(x_k, y_k, m_k) - c_2 \Tilde \H(x_k, y_k, m_k). 
\end{equation}
Here, $\I^\alpha_\ve$ is the mollified irrigation cost defined in (\ref{mic1}), while $\P$ and $\Tilde \H$ are the payoff and penalty functionals defined in (\ref{Hpen}). The optimization is performed over the set of mappings
\begin{equation}\label{mapping}
  \begin{aligned}
& s\mapsto \bigl(x_k(s), y_k(s), m_k(s)\bigr),\quad
s \in [0, 1],\qquad \mbox{subject to the constraints}\\ 
& y_k(s) \ge 0, \quad m_k(s) \ge 0, \quad x_k(0) = y_k(0) = 0.
  \end{aligned}
\end{equation}
Here, $\bigl(x_k(s), y_k(s)\bigr)$ denotes a point  $k$-th branch, while $m_k(s)$ is the density of leaves on the $k$-th branch per unit length. 
To approximate the solution, we consider  continuous and piecewise linear $x_k(\cdot)$, and $y_k(\cdot)$ with break-points $\{t_p\}_{p=0}^{N}\subset[0,1]$, while $m_k(\cdot)$ is piecewise constant function with the same break-points. 
% In our calculations, we shall also need $l_{kp} h_p$ and $m_{kp}\ell_{kp}h_p$. The former is the Euclidean distance between the point $(x_{kp},y_{kp})$ and
% the point $(x_{k,p-1},y_{k,p-1})$.
% \begin{equation}\label{dkp}
% d_{kp}=
% \sqrt{
%     \left(x_{kp}-x_{k,{p-1}}\right)^2
%     +\left(y_{kp}-y_{k,{p-1}}\right)^2
%     }, \qquad\quad  m_d\doteq m\odot d.
% \end{equation}

As a mollifier we choose 
$J(r)~=~\max\{0,\, 1- r^2\}$, and for $s,s'\in [0,1]$, we set
\[R(s,s')\,\doteq\,\sqrt{(x_j(s')-x(s))^2+(y_j(s')-y(s))^2}\,.\]
Then the mollified flux is given by
\begin{equation}\label{Fep}
  \begin{aligned}
  F_\varepsilon(x(s),y(s))&=\sum_{j=1}^n \int_0^1 \frac{1}{\varepsilon}\,
J\!\left(\frac{R(s,s')}{\varepsilon}\right)\, f_j(s')\, \ell_j(s')\,ds'\\
&=\sum_{j=1}^n \int_0^1 \frac{1}{\varepsilon}\,
\max\!\left(0,\;1-\frac{R(s,s')^2}{\varepsilon^2}\right)\, f_j(s')\, \ell_j(s')\,ds'.
\end{aligned}
\end{equation}
Here, $\ell_k$ is the arc-length along the $k$-th branch
and $f_k$ is the flux of water and nutrients from the root to various points along the $k$-th branch. 
% \begin{equation}\label{flux} f_k(s)~=~\int_s^1 m_k(\sigma)\,  d\sigma.\end{equation}
Recall the functionals in the definition of $\J$: the {\bf payoff} functional (the total mass of leaves) given by
\(\Tilde \H=\sum_{k=1}^n f_k(0)\); the {\bf penalization} for cramping too many leaves in the same place is
\begin{equation*}\mathcal{P}~=~\sum_{j,k=1}^n \int_0^1\left(\int_0^1 \exp\Big\{ - \bigl|x_j(s)-x_k(s')\bigr|^2 - \bigl|y_j(s)-y_k(s') \bigr|^2\Big\}\, m_j(s)\, ds\right) \,m_k(s') ds',\end{equation*}
and the {\bf mollified irrigation cost} with discount $\alpha\in \,]0,1[\,$ is 
\begin{equation}\label{iaep}
\mathcal{I}^\alpha_\epsilon~=~\sum_{k=1}^n \int_0^1 \Big[ F_\varepsilon\bigl(x_k(s), y_k(s)\bigr)]^{\alpha-1} f_k(s)\,\ell_k(s)\, ds.
\end{equation}
We use the midpoint rule to compute the integral in the above definition of $\mathcal{I}^\alpha_\epsilon$. For the specific mollifier we have chosen, since 
$x_j(s')$, $y_j(s')$ are piece-wise linear, all integrals
arising in the definition of the total flux $F_\varepsilon(x(s),y(s))$ in equation~\eqref{Fep} can be evaluated exactly at the midpoint. More precisely, on the $j$-th branch and $p$-th interval, for $s_j=\frac{t_{j,p}+t_{j,p-1}}{2}$ the integral w.r.t.~$s'$ in~\eqref{Fep} is evaluated exactly. We then have the following gradient descent algorithm with re-discretization to compute a local minimizer of $\J$.\v
%\begin{algorithm}[H]
%\caption{Gradient descent with
%re-discretization}\label{alg:gd}
%\begin{algorithmic}[1]
\begin{tabular}{|| l ||} \hline
{\bf Algorithm 1.} Gradient descent with re-discretization\\
\hline
{\bf Require:}  Constants $\alpha\in(0,1)$, $c_1,c_2>0$,
  $\varepsilon>0$; initial step size $\tau>0$; max iterations
  $J_{\max}$.\\
{\bf Ensure:}  $(x^*,y^*,m^*)$ approximately
  minimizing~\eqref{cost} with $m_{kp}\ge 0$,
  $y_{kp}\ge 0$.\\
1. Initialise branches as straight lines from the origin;
  set $m_{kp}=m_{\mathrm{init}}$\\
2. {\bf for}  $j=0,1,\ldots,J_{\max}$ ~{\bf do}\\
 3.  $g\leftarrow$
    {\scshape ComputeGradient} $(x,y,m,\alpha,\varepsilon,c_1,c_2)$\\
    % \Comment{Algorithm~\ref{alg:grad}}
 4. Backtracking line search: find the largest
    $\tau_j\le\tau$ such that
    $\mathcal{F}(w-\tau_j\,g)<\mathcal{F}(w)$\\
 5. $w\gets w - \tau_j\,g$; \;
    $y_{kp}\gets\max(y_{kp},0)$, \;
    $m_{kp}\gets\max(m_{kp},0)$\\
  6. Re-discretize: redistribute the $N$
    knots at equal arc-length spacing by linear
    interpolation\\
    7. {\bf end for}\\ 
8. {\bf return} $(x^*,y^*,m^*)$\\
\hline
\end{tabular}
\v 
Figures~\ref{f:opt14} and~\ref{f:optim12} show two  simulations, with $n=11$ and $n=15$ branches, respectively.  As the mollification parameter $\ve$ becomes smaller, it is observed that in an optimal configuration the branches
stick together for a longer time.

%
%\begin{figure}[ht]
%\centerline{\includegraphics[width=3.3cm]{init11.pdf}\quad \includegraphics[width=13cm]
%{tree_opt-11-04.pdf}}
%\caption{\small A simulation with 11 branches.  
%Here $\alpha=0.4$, $c_1=0.4$, $c_2 = 1.4$.
%The left figure shows the initial configuration.
%The other three figures show different stages of the minimization algorithm, where
% the mollification parameter takes the values $\ve=0.5$, $\ve= 0.1$, $\ve=0.03$.
%}
%\label{f:opt14}
%\end{figure}
%
%\begin{figure}[ht]
%\centerline{\hbox{\includegraphics[width=3.3cm]{init15.pdf}}
%\hbox{\includegraphics[width=13cm]{tree_opt-15-05.pdf}}}
%\caption{\small A simulation with 15 branches. Here $\alpha=0.5$, $c_1=0.5$, $c_2 = 1.5$. The left figure shows the initial configuration.
%The other three figures show different stages of the minimization algorithm, where
%the mollification parameter takes the values $\ve=0.8$, $\ve= 0.1$, $\ve=0.01$.
%}
%\label{f:optim12}
%\end{figure}

\begin{figure}[ht]
\centerline{\hbox{\includegraphics[width=4cm]{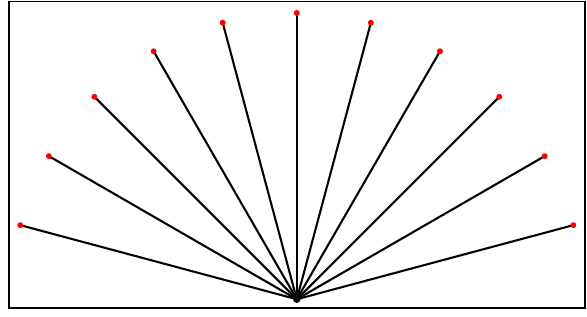}}
\hbox{\includegraphics[width=4cm]{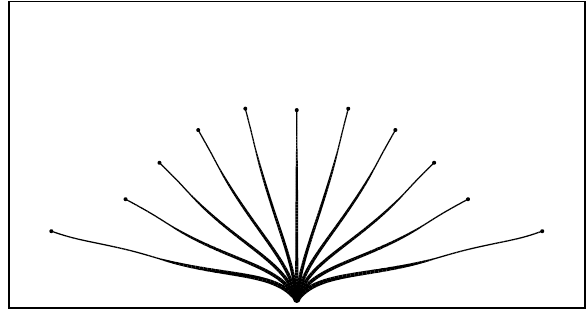}}
\hbox{\includegraphics[width=4cm]{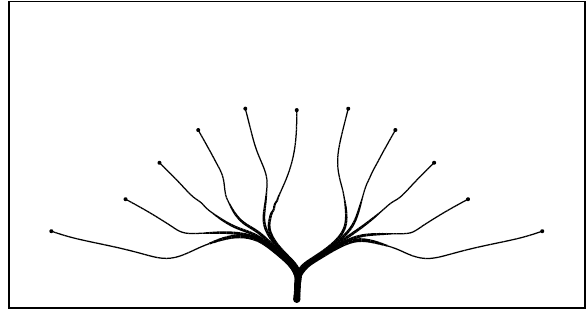}}
\hbox{\includegraphics[width=4cm]{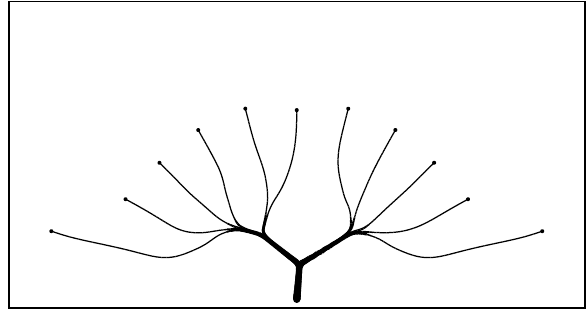}}
}
\caption{\small A simulation with 11 branches.  
Here $\alpha=0.4$, $c_1=0.4$, $c_2 = 1.4$.
The left figure shows the initial configuration.
The other three figures show different stages of the minimization algorithm, where
 the mollification parameter takes the values $\ve=0.5$, $\ve= 0.1$, $\ve=0.03$.
}
\label{f:opt14}
\end{figure}

\begin{figure}[ht]
\centerline{\hbox{\includegraphics[width=4cm]{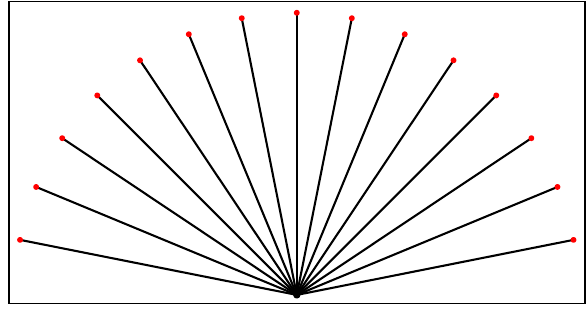}}
\hbox{\includegraphics[width=4cm]{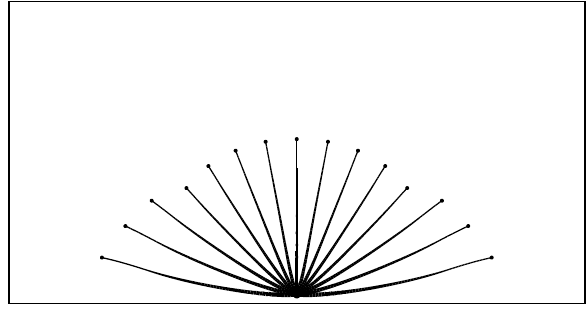}}
\hbox{\includegraphics[width=4cm]{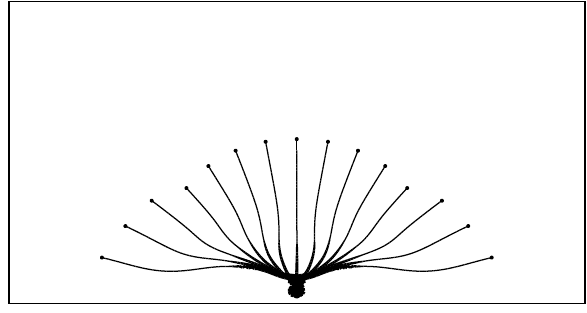}}
\hbox{\includegraphics[width=4cm]{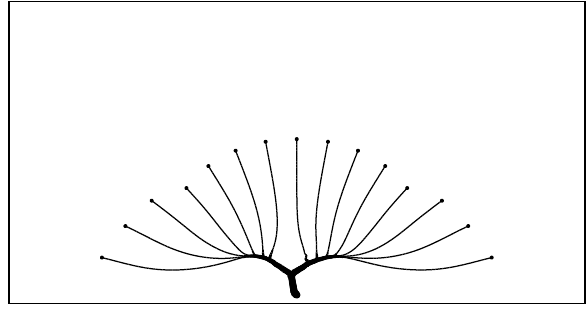}}
}
\caption{\small A simulation with 15 branches. Here $\alpha=0.5$, $c_1=0.5$, $c_2 = 1.5$. The left figure shows the initial configuration.
The other three figures show different stages of the minimization algorithm, where
the mollification parameter takes the values $\ve=0.8$, $\ve= 0.1$, $\ve=0.01$.
}
\label{f:optim12}
\end{figure}

\vs

\section*{Acknowledgments}
The work of Zikatanov was supported in part by the Center for Advanced Studies at the Norwegian Academy of Science and Letters, Norway under the program Mathematical Challenges in Brain Mechanics (2025/2026).

\end{document}